\documentclass[12pt]{amsart}
\usepackage[active]{srcltx}
\usepackage{amsmath}
\usepackage{amsthm}
\usepackage{amssymb}
\usepackage{fullpage}
\usepackage[all,cmtip]{xy}
\usepackage{amsfonts}
\usepackage{verbatim}
\usepackage{amscd,latexsym}
\usepackage{tikz}
\usepackage{caption}
\usepackage{array}
\usepackage{multirow}
\usepackage{mathrsfs}
\usepackage{hyperref}
\usepackage{tabularx}
\usepackage{pgfplots}
\usepackage{float}
\usepackage{bm}
\usepackage{graphicx}
\usepackage{makecell}
\usepackage{longtable}
\usepackage{mathabx}
\usepackage{tabstackengine}
\usepackage{mathtools}
\usepackage{array}
\newsavebox{\mybox}
\usepackage{comment}
\usepackage[utf8]{inputenc}
\usepackage{nicematrix}
\usepackage{kbordermatrix}
\usepackage[T1]{fontenc}
\usepackage{enumitem}
\usepackage{booktabs}
\usepackage{bigints}
\usepackage{chngcntr}
\usepackage{apptools}

\hypersetup{colorlinks=true, linkcolor=blue, filecolor=magenta, urlcolor=cyan, citecolor=purple}
\newtheorem{theorem}{Theorem}
\numberwithin{theorem}{section}
\newtheorem{corollary}[theorem]{Corollary}
\newtheorem{proposition}[theorem]{Proposition}
\newtheorem{lemma}[theorem]{Lemma}
\theoremstyle{definition}
\numberwithin{theorem}{section}
\newtheorem{remark}[theorem]{Remark}
\newtheorem{definition}[theorem]{Definition}

\newcommand{\bt}{\begin{theorem}}
\newcommand{\et}{\end{theorem}}
\newcommand{\bd}{\begin{definition}}
\newcommand{\ed}{\end{definition}}

\newcommand{\ind}{{\rm{ind}}}

\newcommand{\End}{{\rm{End}}}

\tikzset{
	node/.style={
		circle,
		fill,
		inner sep=0.5pt
	},
	font=\footnotesize,
	every pin edge/.style={
		very thin,
		shorten <=1.5pt
	}
}
\usepackage[margin=2.2cm,top=2.5cm,headheight=16pt,headsep=0.3in,heightrounded]{geometry}
\usepackage{fancyhdr}
\def\VR{\kern-\arraycolsep\strut\vrule &\kern-\arraycolsep}

\numberwithin{equation}{section}

\title {On a Bruhat decomposition related to the Shalika subgroup  of $GL(2n)$}

\author{C. Harshitha}
\address{C. Harshitha, Department of Mathematics, Indian Institute of Science Education and Research Tirupati, Srinivasapuram,
	Yerpedu, Tirupati, Andhra Pradesh, 517619,
	India}
\email{harshitha.c@students.iisertirupati.ac.in}

\author{C. G. Venketasubramanian}
\address{C. G. Venketasubramanian, Department of Mathematics, Indian Institute of Science Education and Research Tirupati, Srinivasapuram,
	Yerpedu, Tirupati, Andhra Pradesh, 517619,
	India}
\email{venketcg@iisertirupati.ac.in}

\begin{document}

\subjclass[2020]{Primary 20G05; Secondary 22E50}
\keywords{General Linear Group, Double cosets, Principal series representations, Twisted Jacquet module}
	
\maketitle

\begin{abstract}  
{\footnotesize  Let $F$ be a non-archimedean local field or a finite field.  In this article, we obtain an explicit and  complete set of double coset representatives for $S\backslash GL_{2n}(F)/Q$ where $S$ is the Shalika subgroup and $Q$ a maximal parabolic subgroup of the group $GL_{2n}(F)$ of invertible $2n\times 2n$ matrices.  We compute the cardinality of $S\backslash GL_{2n}(F)/Q$  and also give an alternate perspective on the double cosets arising intrinsically from certain subgroups which are relevant for applications in representation theory.  Finally, if $Q$ is a maximal parabolic subgroup of the type $(r,2n-r),$ we prove that $S\backslash GL_{2n}(F)/Q$ is in one to one correspondence with $\Delta S_n\backslash S_{2n}/S_{r}\times S_{2n-r}$ leading to a Bruhat decomposition.  The results and proofs discussed in this article are valid over any arbitrary field $F$ even though our motivation is from representation theory of $p$-adic and finite linear groups.}
\end{abstract}
\tableofcontents
\section{Introduction}
\subsection{Background and Motivation}
\subsubsection{}
Let $G$ be a group and $H,K$ be any two subgroups of $G.$ A double coset of the pair $(H,K)$ in $G$ is a subset $HgK$ where $g\in G.$ A subset $C$ of $G$ is said to be a complete set of double coset representatives for $(H,K)$ in $G$ if $G$ can be written as a disjoint union of the double cosets $HgK$ where $g\in C.$ Determining a complete set of double coset representatives is not just an interesting problem in its own right,  but also has very important applications in the field of representation theory of groups.

\subsubsection{}
It is a classical result of Mackey \cite{Serre} that the restriction to a subgroup $K$ of a representation of a finite group $G$ induced from a subgroup $H$ can be determined if one knows a complete set of double coset representatives of $(H,K)$ in $G.$ This result of Mackey can be generalized  to smooth representations of $p$-adic groups \cite{BD} if the pair of subgroups $(H,K)$ are both closed in $G.$ We also refer the reader to \cite{BZ2} for the role of double cosets in the Geometric Lemma of Bernstein-Zelevinsky which is an important tool in representations of $p$-adic groups. For a list of some other known cases of determination of double cosets, we refer the reader to \cite{Neretin}.

\subsubsection{}
For the present article, our motivation comes from the study of twisted Jacquet modules (see \cite[\S 2.30]{BZ1}) of representations of $p$-adic groups.  An important class of representations in the category of smooth complex representations of reductive  $p$-adic groups as well as in the category of finite dimensional complex representations of reductive groups defined over finite fields are  principal series representations. A principal series representation (see \cite{BZ1,BZ2}) is an induced representation of the form $\ind_Q^G(\rho)$ where $G$ is the ambient group, $Q$ is  a parabolic subgroup of $G$ with Levi decomposition $Q=LU$ and $\rho$ is a representation of the Levi subgroup $L$ inflated trivially across the unipotent radical $U$ to obtain a representation of $Q.$ 
\subsubsection{}
A problem which has attracted much attention especially in recent times is to compute the structure of twisted Jacquet modules of  irreducible smooth representations of $p$-adic groups as well as those of irreducible complex representations of finite linear groups \cite{ KHJNT,DPDegIMRN,DPDegTIFR,SV}.  In the context of computing the twisted Jacquet module of non-cuspidal irreducible representations, a first step is to compute the twisted Jacquet module of  principal series representations.  An essential component to achieving this is to determine a complete set of double coset representatives for the space $S\backslash G/Q$ where $G$ is the ambient group, $Q$ is the parabolic subgroup from which the principal series representation is parabolically induced from and $S$ is the stabilizer in a parabolic subgroup $P$ with Levi decomposition $P=MN$ of a character $\psi$ of $N,$ with respect to which the twisted Jacquet module is determined.

\subsubsection{}
In the present work, we obtain a complete set of double coset representatives for $S\backslash G /Q$ where $G$ is the group $GL_{2n}(F),$ $Q$ is a maximal parabolic subgroup $P_{r,2n-r}$ of $GL_{2n}(F)$ associated to the partition $(r,2n-r)$ of $2n$ and $S$ is the Shalika subgroup of $GL_{2n}(F).$
We also count the number of such double cosets.  We wish to point out that in the case of an archimedean local field (i.e., where $F$ is either $\mathbb{R}$ or  $\mathbb{C}$), such a double coset decomposition is obtained in \cite[Propisition 4.1]{Geng} while studying generalized Shalika periods, where the parabolic $Q$ is a general standard parabolic subgroup of $GL_{2n}(F)$. The proof of \cite[Proposition 4.1]{Geng} is based on the general theory of reductive algebraic groups and their root systems. In comparison, our proofs are coordinate free, elementary and independent of the choice of the base field $F,$ even though our motivation comes from representation theory of groups defined over $p$-adic and  finite fields. Also, our double coset representatives lead to a succinct  Bruhat decomposition of $GL(2n)$ with respect to the pair $(S,Q).$ All our results are valid over any arbitrary field $F$ and in particular, for the purposes of the results in this article the reader may take $F$ to be any field.

\subsection{Statements of Main Results}
\subsubsection{}
We shall set some basic notations so as to state the main result of our article. Let $F$ be a non-archimedean local field or  a finite field. The $F$-vector space consisting of all $m\times n$ matrices shall be denoted by   $\mathcal{M}_{m,n}(F).$ If $m=n,$  we shall write $\mathcal{M}_{n}(F)$ instead of $\mathcal{M}_{n,n}(F).$  Let $GL_{n}(F)$ denote the group of invertible matrices in $\mathcal{M}_{n}(F).$ Let $r$ be an integer such that $1\leq r < 2n.$ The maximal standard parabolic subgroup of $GL_{2n}(F)$ associated to the partition $(r,2n-r)$ of $2n$ is  \begin{equation}\label{maxPr}
	P_{r,2n-r} = \left\{ \begin{pmatrix}  g_1 & x\\ 0 & g_2 \end{pmatrix} : g_1 \in GL_r(F) , g_2 \in GL_{2n-r}(F), x \in \mathcal{M}_{r,2n-r}(F) \right\}.
\end{equation}
The group $P_{n,n}$ shall be denoted by $P$.  Let

\begin{equation}\label{definitionSpsi}
	S = \left\{ \begin{pmatrix}
		g & x\\ 0 & g 
	\end{pmatrix} : g \in GL_n(F), x \in \mathcal{M}_n(F)\right\}.
\end{equation}  
 The subgroup $S$ is called the Shalika subgroup of $GL_{2n}(F)$ in the literature.

 \subsubsection{}
 In this article, we give a coordinate free approach to determining a complete set of double coset representatives for $S\backslash GL_{2n}(F)/P_{r,2n-r}.$  These results are used  in \cite{HV2} to compute the structure of twisted Jacquet modules of principal series representations of the group $GL_{2n}(F).$ Our work is motivated by that of D. Prasad (\cite{DPDegIMRN} and \cite[Propsoition 7.1]{GS}) obtained  in the case of $GL_4(F)$ where similar double cosets were determined to compute the twisted Jacquet module of a principal series representation of $GL_4(F)$ induced from the maximal parabolic subgroup $P_{2,2}.$

 \subsubsection{}We now state the first main result of our article.
\begin{theorem}\label{mainintro}
	Let $n$ and  $r$ be integers such that $1\leq r<2n.$ For integers $k$ and $l$ such that $\max\{0,r-n\} \leq k \leq \min\{r,n\}$ and $\max\{0,r-n\}\leq l \leq \min\{k,r-k\}$	define $w_{k,l}\in GL_{2n}(F)$ by
\begin{equation}\label{definitionwklintro}
	w_{k,l}=\begin{bmatrix}
		I_k & 0& 0& 0&0 &0 \\
		0&0 & 0& I_{n-k}&0 & 0\\
		0& I_l &0 &0 &0 &0 \\
		0&0 &0 &0 & I_{k-l}&0 \\
		0& 0& I_{r-(k+l)}&0 &0 &0 \\
		0& 0&0 &0 &0 & I_{n-r+l}
	\end{bmatrix}.
\end{equation}
Then,
\begin{equation}\label{doublecosetwkl}
	\left\{w_{k,l}: \max\{0,r-n\}\leq k \leq \min\{r,n\}, \max\{0,r-n\}\leq l \leq \min\{k,r-k\} \right\}
\end{equation} is a complete set of $(S,P_{r,2n-r})$-double coset representatives in $GL_{2n}(F).$
\end{theorem}

\subsubsection{}The number of double coset representatives in Theorem \ref{mainintro} is given by the following result. For an integer $r,$ let $\lfloor \frac{r}{2} \rfloor$ denote the largest integer not exceeding $r/2.$

\begin{theorem}\label{numberintro}
	Let $n$ and  $r$ be integers such that $1\leq r<2n.$  Let $N(n,r)$ denote the cardinality of $S\backslash GL_{2n}(F)/P_{r,2n-r}.$
	Put $\alpha=\max\{0,r-n\}, \gamma=\min\{r,n\}$ and $\beta=\lfloor \frac{r}{2} \rfloor.$ Then,
	\begin{equation*}
		N(n,r)=\frac{(\beta-\alpha+1)(\beta-\alpha+2)+(\gamma-\beta)(\gamma-\beta+1)}{2}.
	\end{equation*}
\end{theorem}
\subsubsection{}
We now present a Bruhat decomposition for the double coset space $S\backslash GL_{2n}(F)/P_{r,2n-r}.$ Let $S_{n}$ denote the permutation group of  $\{1,\dots, n\}.$ Let $\Delta S_n$ be the subgroup of $S_{2n}$ consisting of those permutations $\sigma$ such that $1\leq \sigma(j)\leq n$  and $\sigma(n+j)=n+\sigma(j)$ for  $1\leq j \leq n.$ Let $S_r$  and $S_{2n-r}$ be regarded as subgroups of $S_{2n}$ which permutes $\{1,\dots r\}$ and  $\{r+1,\dots, 2n\}$ respectively.  Let $w_{k,l}'\in S_{2n}$ be defined as follows: $w_{k,l}'$ maps $j\mapsto j$ for $1\leq j \leq k$ and $n+r-l+1\leq j\leq 2n,$ $k+j \mapsto n+j$ for $1\leq j\leq l,$ $k+l+j\mapsto n+k+j$ for $1\leq j \leq r-(k+l),$  $r+j\mapsto k+j$ for $1\leq j\leq n-k$ and $n+r-k+j\mapsto n+k+j$ for $1\leq j\leq k-l.$  If we regard $S_{2n}$ as a subgroup of $GL_{2n}(F)$  as permutation matrices, the permutation $w_{k,l}'$ corresponds to the matrix $w_{k,l}.$ We prove the following theorem.
\begin{theorem}\label{doublecosets-symmetric-intro}
	Let $n$ and  $r$ be integers such that $1\leq r<2n.$  For integers $k$ and $l$ such that $\max\{0,r-n\} \leq k \leq \min\{r,n\}$ and $\max\{0,r-n\}\leq l \leq \min\{k,r-k\}$	define $w_{k,l}\in GL_{2n}(F)$ by \eqref{definitionwklintro}.
	Then, $w_{k,l}\in S_{2n}$  and
	\begin{equation*}\label{doublecosetwkl-symmetric}
		\left\{w_{k,l}: \max\{0,r-n\}\leq k \leq \min\{n,r\}, \max\{0,r-n\}\leq l \leq \min\{k,r-k\} \right\}
	\end{equation*} is a complete set of $(\Delta S_n ,S_{r}\times S_{2n-r})$-double coset representatives in $S_{2n}.$ In particular, we have a bijection 
	\begin{equation}\label{Bruhat}
		S\backslash GL_{2n}(F) /P_{r,2n-r} \overset{1:1}{\longleftrightarrow}\Delta S_n\backslash S_{2n}/S_r \times S_{2n-r}.
	\end{equation}
\end{theorem}
\smallskip

\section{Preliminaries}\label{NotationsPrelims}
\subsection{Notations}\label{Notations}
\subsubsection{}
We shall collect some notations and terminology (see  \cite[\S 7.1]{Garret}) in this section which we shall be using throughout this article. Let $F$ be any field. We recall that $\mathcal{M}_{m,n}(F)$ denotes the $F$-vector space of all $m \times n$ matrices over $F$ and $ GL_{n}(F)$ denotes the group of all invertible matrices of order $n$ over $F.$  Also,  $\mathcal{M}_{n,n}(F)$ will be denoted in short  by $\mathcal{M}_n(F)$ and $I_n$ will denote the identity matrix of order $n.$ 
\subsubsection{}
Let $V$ be a  vector space over $F$ with dimension $n.$  For an integer $k$ such that $1\leq k <n,$ the set of all $k$ dimensional subspaces of $V$  shall be denoted by $Gr(k,V).$  Define a flag $\mathcal{F}$ in $V$ to be a strictly increasing sequence of subspaces  $V_0 \subset \cdots \subset V_m = V.$  
Let $GL(V)$ denote the group consisting of all  invertible linear maps on $V.$ The subgroup of  $GL(V)$ which stabilizes a flag $\mathcal{F}$ is called a parabolic subgroup of $GL(V)$ associated to the flag $\mathcal{F}.$ In particular, the group $GL(V)$ acts transitively on  $Gr(k,V).$ A maximal parabolic subgroup of $GL(V)$ is the stabilizer of a flag  $W\subset V$ where $W\in Gr(k,V).$ 

\subsubsection{}
 Let $F^{n}$ denote the $n$-dimensional vector space over $F$  with its standard basis $\{e_1,  \ldots , e_{n}\}.$ The group $GL_{n}(F)$ acts transitively on $Gr(r,F^n)$ and the stabilizer in $ GL_{n}(F)$ of $\langle e_1, \ldots , e_r \rangle $  is called the  maximal standard parabolic subgroup  associated to the partition $(r,n-r)$ of $n,$ denoted by $P_{r,n-r}.$  Recall from \eqref{maxPr} that
  \begin{equation*}
	P_{r,n-r} = \left\{ \begin{pmatrix}  g_1 & x\\ 0 & g_2 \end{pmatrix} : g_1 \in GL_r(F) , g_2 \in GL_{n-r}(F), x \in \mathcal{M}_{r,n-r}(F) \right\}.
\end{equation*}
Let \begin{equation}\label{maxMr}
M_{r,n-r} = \left\{ \begin{pmatrix}  g_1 & 0\\ 0 & g_2 \end{pmatrix} : g_1 \in GL_r(F) , g_2 \in GL_{n-r}(F) \right\}
\end{equation}

and

\begin{equation}\label{maxNr}
	N_{r,n-r} = \left\{ \begin{pmatrix}  I_r & x\\ 0 & I_{n-r} \end{pmatrix} :  x \in M_
{r,n-r}(F) \right\}.
\end{equation}

\subsubsection{} We have the Levi decomposition $P_{r,n-r} = M_{r,n-r} N_{r,n-r},$ where   $M_{r,n-r} $ and $N_{r,n-r}$ are respectively called the Levi subgroup and  the unipotent radical of $P_{r,n-r}.$   We shall denote the maximal parabolic subgroup $P_{n,n}$ of $GL_{2n}(F)$  by $P,$ its Levi subgroup $M_{n,n}$ by $M$ and its unipotent radical $N_{n,n}$ by $N.$  Put
\begin{equation}\label{definitionMpsi}
	\Delta GL_n(F) = \left\{ \begin{pmatrix}
		g & 0\\ 0 & g 
	\end{pmatrix} : g \in GL_n(F)\right\}.	
\end{equation}  
One has  a semidirect product $S=\Delta GL_n(F)\ltimes N.$

\section{Double cosets}\label{mainresults}
The aim of this section is to prove Theorem \ref{mainintro} in a coordinate free approach which is  achieved in  Theorem \ref{PSpsidoublecosets}.  Throughout this section, we fix $V$ to be a $2n$-dimensional vector space over $F$. 
\subsection{Few actions}
\subsubsection{} 
We begin with the following general lemma.
\begin{lemma}\label{Goribts-doublecosets}
	Let G be a group and $H,K$ be subgroups of G. The double coset space $H\backslash G/K$ is in bijective correspondence with the $G$-orbits of $G/H \times G/K$ under the natural action of $G$ on $G/H\times G/K.$
\end{lemma}
\begin{proof}
In fact, the map $HgK \mapsto {\rm{Orb}}_{G}(H,gK)$  is well defined and is a bijection. We omit the details of the proof.
\end{proof}

\subsubsection{}
For a subspace $W'\in Gr(r,V),$  denote the stabilizer in $GL(V)$ of $W'$  by $P_{W'}.$ Then, $P_{W'}:=\{f \in GL(V): f(W')=W'\}$  is a maximal parabolic subgroup of $GL(V).$  We may identify $GL(V)/P_{W'}$ with $Gr(r,V)$ via $gP_{W'}\leftrightarrow g(W').$ 

\subsubsection{}
We now introduce a space $X$ which holds the key to establishing Theorem \ref{mainintro}. Define
\begin{equation}
	X:=\left\{(W,j): W \in Gr(n,V), j:W \to V/W \:\text{is an isomorphism}\right\}.
\end{equation}
Define an action of  $GL(V)$ on $X$ by
\begin{equation}\label{GLactionforSpsi}
g\cdot(W,j)=(g(W),\bar g\circ j\circ g^{-1}), 
\end{equation}
for $g \in GL(V)$ and $(W,j) \in X,$ where $\bar{g}$ is the map induced by $g$ from $V/W$ to $V/g(W).$  It is straight forward to verify that \eqref{GLactionforSpsi} is indeed an action. We shall next show that this action is transitive.

\begin{lemma}\label{gmodSpsi}
The action of $GL(V)$ on $X$ given by \eqref{GLactionforSpsi} is transitive.

\end{lemma}

\begin{proof}
	Let $(W,j),(W',j')\in X.$
 Choose a basis $\{v_1,\dots,v_n\}$ of $W.$ For $1\leq i \leq n,$ choose  $v_{n+i}\in V$ such that $j(v_i)=v_{n+i}+W$ so that $\{v_{n+1}+W,\dots,v_{2n}+W\}$ is a basis of $V/W.$ Similarly, choose a basis $\{v_1',\dots,v_n'\}$ of $W'$ and choose $\{v_{n+i}':1\leq i \leq n \}\subset V$ such that  $j'(v_i')=v_{n+i}'+W'$ for $1\leq i \leq n,$ so that $\{v_{n+1}'+W',\dots,v_{2n}'+W'\}$ is a basis of $V/W'.$
 It follows that both $\{v_i:1\leq i\leq 2n\}$ and $\{v_i':1\leq i\leq 2n\}$ are bases of $V.$ Define $g:V\to V$ by setting $g(v_i)=v_i'$ for $1\leq i \leq 2n.$  Clearly, $g\in GL(V),g(W)=W'$ and for $1 \leq i \leq n,$ we have $\bar g\circ j\circ g^{-1}(v_i')=\bar g\circ j(v_i)=\bar g(v_{n+i}+W)=v_{n+i}'+W'=j'(v_i').$ This completes the proof.
\end{proof}

\subsubsection{}
Fix any $(W,j)\in X.$ Put
\begin{equation}
S_{(W,j)}=	\{g\in P_W: \bar g\circ j=j\circ g\}.
\end{equation}
Then,  $S_{(W,j)}$ is the stabilizer of $(W,j)$  in $GL(V).$  By Lemma \ref{gmodSpsi}, we can identify $GL(V)/S_{(W,j)}$ with $X.$  In view of Lemma  \ref{Goribts-doublecosets}, to obtain a complete set of double coset representatives for $S_{(W,j)}\backslash GL(V) /P_{W'},$ it is sufficient to consider the orbits under the action of $GL(V)$ on $X\times Gr(r,V)$ given by 
\begin{equation}\label{GLaction-product}
	g\cdot ((W,j),W')=((g(W),\bar g\circ j\circ g^{-1}),g(W')),
\end{equation} 
for $g\in GL(V), ((W,j),W') \in X\times Gr(r,V).$ 

\subsection{Main results}
We shall establish the coordinate free version of the main result Theorem \ref{mainintro} of our article in this subsection. 
\subsubsection{}
We need the following lemma provides a necessary condition for a pair of members in $X\times Gr(r,V)$ to be in the same $GL(V)$-orbit under the action \eqref{GLactionforSpsi}.

\begin{lemma}\label{dimensionspreserved}
	Suppose $((W_1,j),W), ((W_1',j'),W')\in X\times Gr(r,V) $ and $g\in GL(V)$ are such that
	$g\cdot ((W_1,j),W)=((W_1',j'),W').$ Then the following hold:
	\begin{enumerate}
		\item $\dim(W\cap W_1)=\dim(W'\cap W_1'),$ and
		\item $\dim(j(W\cap W_1)\cap (W+W_1)/W_1)= \dim(j'(W'\cap W_1')\cap (W'+W_1')/W_1').$
	\end{enumerate}  
\end{lemma}
\begin{proof}
Since $g(W)=W', g(W_1)=W_1'$ and  $g$ is injective,  $g(W\cap W_1)=W'\cap W_1',$ proving (1).  To prove (2), we note that it is sufficient to prove the following claim:
\begin{equation}\label{midlemma}
\bar g(j(W\cap W_1)\cap (W+W_1)/W_1)\subset j'(W'\cap W_1')\cap (W'+W_1')/W_1'.
\end{equation}
	
Accepting \eqref{midlemma} and noting that $\bar{g}:V/W_1 \to V/W_1'$ is an isomorphism, it follows that $\dim(j(W\cap W_1)\cap (W+W_1)/W_1)\leq \dim(j(W'\cap W_1')\cap (W'+W_1')/W_1').$ Swapping $g,W,W_1,j$ with $g^{-1}, W', W_1',j'$ gives the inequality in the reverse direction proving (2).

To prove \eqref{midlemma}, let $v+W_1 \in j(W\cap W_1)\cap (W+W_1)/W_1.$ Suppose $w\in W\cap W_1$ is such that $v+W_1=j(w),$ then $\bar g(v+W_1)=(\bar g\circ j)(w)=(j'\circ g)(w).$ But, $g(w)\in W'\cap W_1'$ and it gives $\bar g(v+W_1) \in j'(W'\cap W_1').$ Also, $\bar g (v+W_1)=g(v)+W_1' \in (W'+W_1')/W_1'.$ This proves \eqref{midlemma}.\end{proof}

\subsubsection{}
We are now ready to state and prove  the main theorem of this section which establishes the coordinate free version of Theorem \ref{mainintro}.

\begin{theorem}\label{PSpsidoublecosets}
Given an element $((W_1,j),W)\in X\times Gr(r,V),$ its orbit under the action of $GL(V)$  given by \eqref{GLaction-product} is determined by 
\begin{enumerate}
\item $\dim(W\cap W_1),$ and
\item $\dim(j(W\cap W_1)\cap (W+W_1)/W_1).$
\end{enumerate} 
\end{theorem}
\begin{proof}
Suppose $((W_1,j),W), ((W_1',j'),W') \in X\times Gr(r,V) $ are such that 
$\dim(W\cap W_1)=k=\dim(W'\cap W_1')$ and $\dim(j(W\cap W_1)\cap (W+W_1)/W_1)=l= \dim(j'(W'\cap W_1')\cap (W'+W_1')/W_1').$ To prove the theorem,  by Lemma \ref{dimensionspreserved}, it is sufficient to show that  there exists a $g \in GL(V)$ such that $g\cdot ((W_1,j),W)=((W_1',j'),W').$ The strategy of our proof is to construct a basis $\{v_1,\dots,v_{2n}\}$ of $V$ such that the following properties hold:
	\begin{enumerate}
		\item[(a)] $\{v_1,\dots,v_n\}$ is a basis of $W_1$ and  $\{v_{n+1}+W_1,\dots,v_{2n}+W_1\}$ is a basis of $V/W_1,$
		\item[(b)] $\{v_1,\dots,v_k\}\cup \{v_{n+1},\dots,v_{n+l}\}\cup\{v_{(n+k)+1},\dots, v_{(n+k)+r-k-l}\}$ is a basis of $W,$ and
		\item[(c)]  for $1\leq i \leq n,$  $j(v_i)=v_{n+i}+W_1.$
	\end{enumerate}
Granting such a basis exists as above, we can construct in a similar way, a basis $\{v_1',\dots,v_{2n}'\}$ of $V$ such that:
	\begin{enumerate}
		\item[(a)']$\{v_1',\dots,v_n'\}$ is a basis of $W_1'$ and  $\{v_{n+1}'+W_1',\dots,v_{2n}'+W_1'\}$ is a basis of $V/W_1',$
		\item[(b)'] $\{v_1',\dots,v_k'\}\cup \{v_{n+1}',\dots,v_{n+l}'\}\cup\{v_{(n+k)+1}',\dots, v_{(n+k)+r-k-l}'\}$ is a basis of $W',$ and
		\item[(c)']  for $1\leq i \leq n,$  $j'(v_i')=v_{n+i}'+W_1'.$
	\end{enumerate}
	Given such bases exist, we can define $g:V \to V$ by $g(v_i)=v_i'$ for  $1\leq i\leq 2n.$ Then, $g\in GL(V),$ $g(W)=W'$ and $g(W_1)=W_1'.$ Also, for $1\leq i\leq n,$ one has $\bar g\circ j\circ g^{-1}(v_i')= \bar g\circ j(v_i)=\bar g(v_{n+i}+W_1)=v_{n+i}'+W_1' = j'(v_i'),$ proving $g\cdot ((W_1,j),W)=((W_1',j'),W').$   The proof of the theorem is thus reduced to constructing a basis of $V$ satisfying (a), (b) and (c). 
	
	For the remaining part of this proof, we shall denote a coset $v+W_1$  by $\bar{v}.$ We refer the reader to the diagram below to keep track of the various vector spaces involved in the proof.

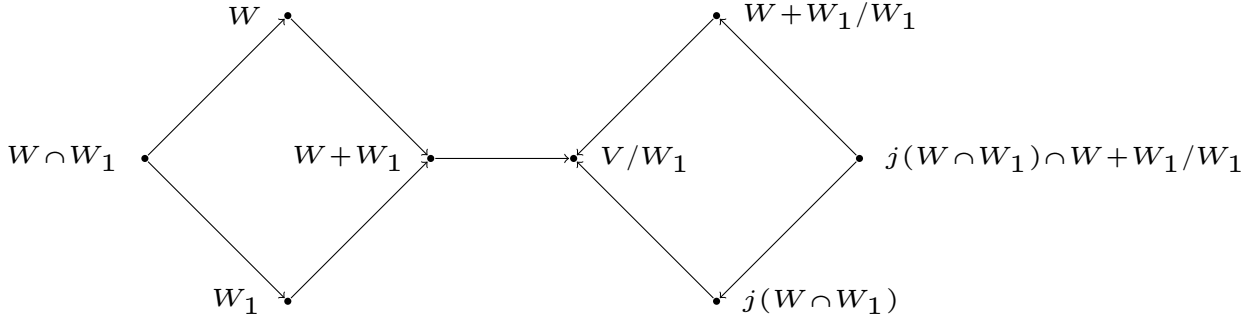
\begin{center}
	\begin{tikzpicture}[scale=1.9, transform shape]
	\node[node, label={left:$\scriptscriptstyle{W\cap W_1}$}] (V1) at (-2,0) {};
	\node[node, label={left:$\scriptscriptstyle{W}$}] (V2) at (-1,1) {};
	\node[node, label={left:$\scriptscriptstyle{W_1}$}] (V3) at (-1,-1) {};
	\node[node, label={left:$\scriptscriptstyle{W+W_1}$}] (V4) at (0,0) {};
	\node[node, label={right:$\scriptscriptstyle{V/W_1}$}] (V5) at (1,0) {};
	\node[node, label={right:$\scriptscriptstyle{W+W_1/W_1}$}] (V6) at (2,1) {};
	\node[node, label={right:$\scriptscriptstyle{j(W\cap W_1)}$}] (V7) at (2,-1) {};
	\node[node, label={right:$\scriptscriptstyle{j(W\cap W_1)\cap W+W_1/W_1}$}] (V8) at (3,0) {};
	\draw[->] (V1) -- (V2);
	\draw[->] (V1) -- (V3);
	\draw[->] (V2) -- (V4);
	\draw[->] (V3) -- (V4);
	\draw[->] (V4) -- (V5);
	\draw[<-] (V5) -- (V6);
	\draw[<-] (V5) -- (V7);
	\draw[<-] (V6) -- (V8);
	\draw[<-] (V7) -- (V8);
	
\end{tikzpicture}	

\end{center}

We note that $\dim((W+W_1)/W_1)=r-k.$  Choose a subset $\{v_{n+1}, \dots,v_{n+l}\}$  of $W$ such that $\{\bar{v}_{n+1}, \dots,\bar{v}_{n+l}\}$ forms a basis of $j(W\cap W_1)\cap (W+W_1)/W_1.$ Extend this to a basis of $(W+W_1)/W_1,$ say, by adjoining $\{\bar{v}_{(n+k)+1},\dots, \bar{v}_{(n+k)+r-k-l}\}$ where $\{v_{(n+k)+1},\dots, v_{(n+k)+r-k-l}\}\subset W.$ Put $v_i=j^{-1}(\bar{v}_{n+i})$ for $1\leq i \leq l.$ Since $j$ is an isomorphism, $\{v_1,\dots, v_l\}$ is a linearly independent set in $W\cap W_1.$  We extend $\{v_1,\dots, v_l\}$  to a basis $\{v_1,\dots,v_k\}$ of $W\cap W_1$ by adjoining the vectors $v_{l+1},\dots,v_k.$

For $1\leq i \leq r-k-l,$ put $v_{k+i}=j^{-1}(\bar{v}_{(n+k)+i}).$  We claim that $\{v_1,\dots,v_k\} \cup \{v_{k+1},\dots, v_{k+(r-k-l)}\}$ is a linearly independent set in $W_1.$ Since $\{v_1,\dots,v_k\}$ and $\{v_{k+1},\dots, v_{k+(r-k-l)}\}$ are both linearly independent, to prove the claim, it suffices to show that
\begin{equation*}\langle v_1,\dots,v_k\rangle \cap \langle v_{k+1},\dots, v_{k+(r-k-l)}\rangle = \{0\}.
\end{equation*}
To this end, suppose $v\in \langle v_1,\dots,v_k\rangle \cap  \langle v_{k+1},\dots, v_{k+(r-k-l)}\rangle.$  Then,  
 \begin{equation*}j(v) \in \langle \bar{v}_{n+1},\dots,\bar{v}_{n+k}\rangle \cap  \langle \bar{v}_{(n+k)+1},\dots, \bar{v}_{(n+k)+(r-k-l)}\rangle \subset j(W\cap W_1)\cap (W+W_1)/W_1.
 \end{equation*}  
 But, $\{\bar{v}_{n+1},\dots,\bar{v}_{n+l}\}\cup \{\bar{v}_{(n+k)+1},\dots, \bar{v}_{(n+k)+r-k-l}\} $ is a linearly independent set and $\{\bar{v}_{n+1},\dots,\bar{v}_{n+l}\}$ is a basis for $j(W\cap W_1)\cap (W+W_1)/W_1.$  Therefore, we obtain $j(v)=0$ which also yields $v=0$ and proves our claim. 
 
 Now, extend $\{v_1,\dots,v_{r-l}\}$ further to a basis of $W_1$ by adjoining, say, $\{v_{r-l+1},\dots,v_n\}.$ To summarize, we have obtained a basis of $W_1,$ namely, $\{v_1,\dots,v_n\}.$ For $l+1\leq i \leq k$ and $r-l+1 \leq i \leq n,$ we choose $v_{n+i} \in V$ such that $j(v_i)=\bar{v}_{n+i}.$ Thus, $\{v_1,\dots,v_n\}$ is a basis of $W_1$ and by our construction,  $\{j(v_1),\dots,j(v_n)\}= \{\bar{v}_{n+1}, \dots, \bar{v}_{2n}\}$  is a basis of $V/W_1.$ This completes the proof of the theorem.
\end{proof}

\section{Description in coordinates}
In this section,  we establish Theorem \ref{mainintro}, Theorem \ref{numberintro} and Theorem \ref{Bruhat}. Also, we give an alternate perspective on obtaining the double coset representatives $w_{k,l}$ appearing in Theorem \ref{mainintro} using certain specific subgroups of $GL_{2n}(F)$ which arise intrinsically from representation theory in \cite{HV2}. This alternate description also proves to be useful in applications as we have observed in \cite{HV2}. 

Throughout this section, we shall fix $V$ to be the $2n$-dimensional vector space $F^{2n}$ with  standard basis  $\mathcal{B}:=\{e_1,\dots,e_{2n}\}.$  We shall identify $\End(V)$ with $\mathcal{M}_{2n}(F)$ via the isomorphism $f\mapsto [f]_{_\mathcal{B}},$ where $[f]_{_\mathcal{B}}$ denotes the matrix of $f$ with respect to $\mathcal{B}.$  The group $GL_{2n}(F)$ is identified with $GL(V)$ under this correspondence. We shall  make use of this elementary fact throughout this section.

\subsection{Double cosets in Coordinates}
\subsubsection{Bounds of $k$ and $l$}
 We first note the following lemma giving the bounds for the integers $k$ and $l$ appearing in Theorem \ref{mainintro}.

\begin{lemma}\label{boundsofkandl}
	For $((W_1,j),W)\in X\times Gr(r,V),$ put
	 $k=\dim(W\cap W_1),$ and 		 $l=\dim(j(W\cap W_1)\cap (W+W_1)/W_1).$ Then, $k$ and $l$ satisfy the following inequalities:
	\begin{enumerate}
		\item $\max\{0,r-n\} \leq k \leq \min\{r,n\}.$ 
		\item $\max\{0,r-n\} \leq l \leq \min\{k,r-k\}.$
	\end{enumerate}	
\end{lemma}
\begin{proof}
It is trivial to note $0 \leq \dim(W\cap W_1) \leq \min\{\dim(W),\dim(W_1) \}= \min\{r,n\}.$ Also, $\dim(W+W_1)=\dim(W)+\dim(W_1)-\dim(W\cap W_1)\leq 2n$ gives $\dim(W\cap W_1) \geq \dim(W)+\dim(W_1) -2n.$ So, $\dim(W\cap W_1) \geq r+n-2n =r-n$  which gives $\max\{0,r-n\} \leq k \leq \min\{r,n\}.$ 
	
To prove (2), we first note that 
$\dim((W+W_1)/W_1)= \dim(W)+\dim(W_1)-\dim(W\cap W_1)-\dim(W_1)=r-k.$ 
Since $j$ is an isomorphism, $\dim(j(W\cap W_1))=\dim(W\cap W_1)=k.$ Therefore, $l \leq \min\{k,r-k\}.$ Also, $\dim((W+W_1)/W_1+j(W\cap W_1))=\dim((W+W_1)/W_1)+\dim(j(W\cap W_1))-\dim((W+W_1)/W_1\cap j(W\cap W_1))\leq \dim(V/W_1)=n.$ This yields  $l=\dim((W+W_1)/W_1\cap j(W\cap W_1))\geq \dim((W+W_1)/W_1)+\dim(j(W\cap W_1))-n=(r-k)+k-n=r-n.$ This establishes (2) and the lemma.
\end{proof}

\subsubsection{Explicit double coset representatives}
 We have the following corollary to Theorem \ref{PSpsidoublecosets} which gives an explicit set of double coset representatives of $(S,P_{r,2n-r})$ in $GL_{2n}(F).$

\begin{corollary}\label{doublecosets}
	Let $n$ and  $r$ be integers such that $1\leq r<2n.$  For integers $k$ and $l$ such that $\max\{0,r-n\} \leq k \leq \min\{r,n\}$ and $\max\{0,r-n\}\leq l \leq \min\{k,r-k\}$	define $w_{k,l}\in GL_{2n}(F)$ by
	\begin{equation}\label{definitionwkl}
		w_{k,l}=\begin{bmatrix}
			I_k & 0& 0& 0&0 &0 \\
			0&0 & 0& I_{n-k}&0 & 0\\
			0& I_l &0 &0 &0 &0 \\
			0&0 &0 &0 & I_{k-l}&0 \\
			0& 0& I_{r-(k+l)}&0 &0 &0 \\
			0& 0&0 &0 &0 & I_{n-r+l}
		\end{bmatrix}.
	\end{equation}
	Then,
	\begin{equation}\label{doublecosetwkl}
		\left\{w_{k,l}: \max\{0,r-n\}\leq k \leq \min\{n,r\}, \max\{0,r-n\}\leq l \leq \min\{k,r-k\} \right\}
	\end{equation} is a complete set of $(S,P_{r,2n-r})$-double coset representatives in $GL_{2n}(F).$
\end{corollary}
\begin{proof}
Let $W_0'=\langle e_1,\dots,e_r\rangle$ and $W_0=\langle e_1,\dots,e_n\rangle.$ Let $j_0:W_0\to V/W_0$ be the isomorphism defined by $j_0(e_i)=e_{n+i}+W_0$ for $1\leq i\leq n.$ One observes that $P_{W_0'}=P_{r,2n-r}$ and $S_{(W_0,j_0)}=S.$
For each $k,l$ satisfying $\max\{0,r-n\}\leq k \leq \min\{n,r\}, \max\{0,r-n\}\leq l \leq \min\{k,r-k\},$ put 
\begin{equation*}W_{k, l} = \langle e_1, \dots , e_k, e_{n+1}, \dots , e_{n+l}, e_{n+k+1}, \dots, e_{n+k+ (r-(k+l))}\rangle.
\end{equation*}
Clearly, $W_{k,l}\cap W_0=\langle e_1,\dots,e_k\rangle$ and $\dim(W_{k,l}\cap W_0)=k.$ It is also easy to see that $j_0(W_{k,l}\cap W_0)=\langle e_{n+1}+W_0,\dots,e_{n+k}+W_0\rangle$ and $(W_{k,l}+W_0)/W_0=\langle e_{n+1}+W_0, \dots , e_{n+l}+W_0, e_{(n+k)+1}+W_0, \dots, e_{(n+k)+ (r-(k+l))}+W_0\rangle.$ Hence, $j_0(W_{k,l}\cap W_0)\cap (W_{k,l}+W_0)/W_0=\langle e_{n+1}+W_0,\dots,e_{n+l}+W_0\rangle.$ This yields  $\dim(j_0(W_{k,l})\cap (W_{k,l}+W_0)/W_0)=l.$   
By Theorem \ref{PSpsidoublecosets}, 
\begin{equation*}
\{((W_0,j_0),W_{k,l}): \max\{0,r-n\}\leq k \leq \min\{n,r\}, \max\{0,r-n\}\leq l \leq \min\{k,r-k\}\}
\end{equation*}
forms a complete set of orbit representatives for the action of $GL(V)$ on $X\times Gr(r,V).$
		
Let $w_{k,l}$ be as in \eqref{definitionwkl}. Since $w_{k,l}(W_0')=W_{k,l},$  under the identification of $Gr(r,V)$ with $GL_{2n}(F)/P_{r,2n-r},$ the subspace $W_{k,l}$ corresponds to the left coset $w_{k,l}P_{r,2n-r}.$  Similarly, $(W_0,j_0)$ corresponds to $S$ under the identification of $X$ with $GL_{2n}(F)/S.$  Consequently, the orbit representatives for the action of $GL_{2n}(F)$ on $GL_{2n}(F)/S\times GL_{2n}(F)/P_{r,2n-r}$ are given by 
\begin{equation*}\{(S,w_{k,l}P_{r,2n-r}): \max\{0,r-n\}\leq k \leq \min\{n,r\}, \max\{0,r-n\}\leq l \leq \min\{k,r-k\} \}.
\end{equation*} 
We can conclude from the proof of Lemma \ref{Goribts-doublecosets} that a complete set of $(S,P_{r,2n-r})$-double coset representatives in $GL_{2n}(F)$ is given by  
\begin{equation*}
\left\{w_{k,l}: \max\{0,r-n\}\leq k \leq \min\{n,r\}, \max\{0,r-n\}\leq l \leq \min\{k,r-k\} \right\}.\qedhere
\end{equation*}
\end{proof}

\subsubsection{Number of Double cosets}
Note that Corollary \ref{doublecosets} shows in particular that the number of distinct $(S,P_{r,2n-r})$-double cosets in $GL_{2n}(F)$ is finite.  The following theorem counts the number of such double cosets.
\begin{theorem}
	Let $n$ and  $r$ be integers such that $1\leq r<2n.$  Let $N(n,r)$ denote the cardinality of $S\backslash GL_{2n}(F)/P_{r,2n-r}.$
Put $\alpha=\max\{0,r-n\}, \gamma=\min\{r,n\}$ and $\beta=\lfloor \frac{r}{2} \rfloor.$ Then,
\begin{equation}\label{count}
	N(n,r)=\frac{(\beta-\alpha+1)(\beta-\alpha+2)+(\gamma-\beta)(\gamma-\beta+1)}{2}.
\end{equation}

\end{theorem}
\begin{proof}  
If $r\leq n,$ we have $\alpha=0$ and $\gamma=r.$ On the other hand, if $r>n,$ we have $\alpha=r-n$ and $\gamma=n.$ 	For a fixed integer $k$ satisfying $\alpha\leq k \leq \gamma,$ the integer $l$ assumes values between $\alpha$ and $\min\{k,r-k\}.$ Let $l(k)$ denote the number of corresponding $l$ values for a fixed $k.$ We then have the following table:
	
	\begin{center}
		\begin{tabular}{|c|c|c|c|}
			\hline
			Case &Range of $k$ & $\min\{k,r-k\}$ & $l(k)$ \\
			\hline
			\multirow{2}{*}{$r\leq n$}
			& $0\leq k \leq \lfloor r/2 \rfloor$ & $k$ & $k+1$ \\
			\cline{2-4}
			& $\lfloor r/2 \rfloor +1\leq k \leq r$ & $r-k$ & $r-k+1$ \\
			\hline
			\multirow{2}{*}{$n\leq r$}
			& $r-n\leq k \leq \lfloor r/2 \rfloor$ & $k$ & $k-r+n+1$ \\
			\cline{2-4}
			& $\lfloor r/2 \rfloor +1\leq k \leq n$ & $r-k$ & $n-k+1$ \\
			\hline
		\end{tabular}
	\end{center}
	It is easy to observe that $N(n,r)=\sum_{k}l(k).$ To obtain \eqref{count}, we divide the proof into two cases as to when $r\leq n$ and $r>n.$ To this end, if $r\leq n,$ we obtain
	\begin{eqnarray*}
		\sum_k l(k)=\sum_{k=0}^{r}l(k)=&\displaystyle\sum_{k=0}^{\lfloor \frac{r}{2} \rfloor} (k+1) + \sum_{k=\lfloor \frac{r}{2} \rfloor +1}^{r} (r-k+1)\\=&\displaystyle\frac{\left(\lfloor\frac{r}{2}\rfloor+1\right)\left(\lfloor\frac{r}{2}\rfloor+2\right)}{2}+\frac{\left(r-\lfloor\frac{r}{2}\rfloor\right)\left(r-\lfloor\frac{r}{2}\rfloor+1\right)}{2}.
		\end{eqnarray*}
		If $r > n,$ we get
		\begin{eqnarray*}
			\sum_k l(k)=\sum_{k=r-n}^{n}l(k)=&\displaystyle\sum_{k=r-n}^{\lfloor \frac{r}{2} \rfloor} (k-r+n+1) + \sum_{k=\lfloor \frac{r}{2} \rfloor +1}^{n} (n-k+1)\\
			=&\displaystyle\frac{\left( n-r+\lfloor\frac{r}{2}\rfloor+1\right)\left( n-r+\lfloor\frac{r}{2}\rfloor+2\right)}{2}+\frac{\left(n-\lfloor\frac{r}{2}\rfloor\right)\left(n-\lfloor\frac{r}{2}\rfloor+1\right)}{2}.
			\end{eqnarray*}
	In either case,  $\sum_k l(k)=\displaystyle\frac{(\beta-\alpha+1)(\beta-\alpha+2)+(\gamma-\beta)(\gamma-\beta+1)}{2},$ and the proof of the theorem is complete.	
\end{proof}

\subsection{An alternate description}
In the reminder of this section, we give an alternate perspective on viewing the double coset representatives $w_{k,l}$ obtained in Corollary \ref{doublecosets}  in a way that is  relevant for the representation theory in \cite{HV2}.  To achieve this, we introduce few subgroups.  

\subsubsection{}Let $n$ and  $r$ be 
 integers such that $1\leq r<2n.$ To ease our notation, put  $\alpha=\max\{0, r-n\}$ and $\gamma=\min\{r,n\}.$ For $k$ such that $\alpha\leq k \leq \gamma,$ let $J_{k,\alpha}:= w_{k,\alpha}P_{r,2n-r}w_{k,\alpha}^{-1} \cap P.$ We shall write a block matrix $\begin{pmatrix} 
	a & b\\
	c & d
\end{pmatrix}$ where $a\in \mathcal{M}_{n_1,n_3}(F), b\in \mathcal{M}_{n_1,n_4}(F), c\in \mathcal{M}_{n_2,n_3}(F)$ and $d\in \mathcal{M}_{n_2,n_4}(F)$ by  

\renewcommand{\kbldelim}{(}% Left delimiter
\renewcommand{\kbrdelim}{)}% Right delimiter
\[
\kbordermatrix{
	&n_3& n_4\\
	n_1& a & b\\
	n_2& c & d},
\]
specifically indicating the sizes of the blocks. 

\subsubsection{}Using  the above notation,  for any $l$ such that $\alpha \leq l \leq \min\{k,r-k\}, $ we write a matrix $p\in P_{r,2n-r}$  as 
\renewcommand{\kbldelim}{(}% Left delimiter
\renewcommand{\kbrdelim}{)}% Right delimiter
\[
p = \kbordermatrix{
	&k&l& r-k-l& & n-k &k-l&n-r+l\\
	k&g_1&g_2& g_3 & \VR x_1 & x_2&x_3\\
	l&g_4&g_5& g_6 & \VR x_4&x_5& x_6\\
	r-k-l& g_7 & g_8 & g_9& \VR x_7 & x_8 & x_9\\
	\hline
	n-k & 0&0&0& \VR h_1 & h_2 & h_3 \\
	k-l & 0& 0& 0& \VR h_4 & h_5 & h_6\\
	n-r+l & 0& 0& 0& \VR h_7 & h_8 & h_9
}.
\]
We then have
\renewcommand{\kbldelim}{(}% Left delimiter
\renewcommand{\kbrdelim}{)}% Right delimiter
\begin{equation}\label{wklP-conjugate}
	w_{k,l}p  w_{k,l}^{-1}= \kbordermatrix{
		&k&n-k& & l&k-l&r-k-l&n-r+l\\
		k&g_1&x_1& \VR g_2&x_2&g_3&x_3\\
		n-k&0&h_1& \VR 0&h_2&0&h_3\\
		\hline
		l&g_4&x_4& \VR g_5&x_5&g_6&x_6\\
		k-l&0&h_4& \VR 0&h_5&0&h_6\\
		r-k-l&g_7&x_7& \VR g_8&x_8&g_9&x_9\\
		n-r+l&0&h_7& \VR 0&h_8&0&h_9\\
	}.
\end{equation}

\subsubsection{}
It is now easy to see that the subgroup $J_{k,\alpha}$ of $P$ is given by 

\renewcommand{\kbldelim}{(}% Left delimiter
\renewcommand{\kbrdelim}{)}% Right delimiter
\[
J_{k,\alpha}= \left\{\kbordermatrix{
	&k&n-k& & \alpha&k-\alpha&r-k-\alpha&n-r+\alpha\\
	k&\ast&\ast& \VR \ast&\ast&\ast&\ast\\
	n-k&0&\ast& \VR 0&\ast&0&\ast\\
	\hline
	\alpha&0&0& \VR \ast&\ast&\ast&\ast\\
	k-\alpha&0&0& \VR 0&\ast&0&\ast\\
	r-k-\alpha&0&0& \VR \ast&\ast&\ast&\ast\\
	n-r+\alpha&0&0& \VR 0&\ast&0&\ast\\
}\right\}.
\]
\subsubsection{}
For a fixed $k,$  we define a subgroup $N_0$ of $GL_{2n}(F)$ and an element $w\in GL_{2n}(F)$ as follows.
Put

  \renewcommand{\kbldelim}{(}% Left delimiter
\renewcommand{\kbrdelim}{)}% Right delimiter
\[
N_0= \left\{\kbordermatrix{
	&k&n-k& & \alpha&k-\alpha&r-k-\alpha&n-r+\alpha\\
	k&I_k&0& \VR \ast&\ast&\ast&\ast\\
	n-k&0&I_{n-k}& \VR 0&\ast&0&\ast\\
	\hline
	\alpha&0&0& \VR I_{\alpha}&0&0&0\\
	k-\alpha&0&0& \VR 0&I_{k-\alpha}&0&0\\
	r-k-\alpha&0&0& \VR 0&0&I_{r-k-\alpha}&0\\
	n-r+\alpha&0&0& \VR 0&0&0&I_{n-r+\alpha}\\
}\right\},
\]
and 
\begin{equation*}w=\begin{pmatrix}
		I_{\alpha} & 0& 0& 0\\
		0& 0& I_{k-\alpha} &0\\
		0 & I_{r-k-\alpha} & 0 & 0\\
		0& 0& 0& I_{n-r+\alpha}	\end{pmatrix}. \end{equation*}
Also put, \begin{equation}\label{Ptilde}
	\tilde{P}_{r-k,n-r+k}=w{P}_{r-k,n-r+k}w^{-1}.
\end{equation} 
It can then be checked that $N_0$ is a normal subgroup of $J_{k,\alpha}.$ Moreover, we have a semi-direct product
\begin{equation}
	J_{k,\alpha}= (P_{k,n-k}\times \tilde{P}_{r-k,n-r+k})\ltimes N_0.
\end{equation} 

\subsubsection{}
We shall need the following lemma. 
	\begin{lemma}\label{lalternate}
		Let $r$ and $n$ be integers such that $1\leq r<2n.$ Put $\alpha=\max\{0,r-n\},\gamma=\min\{r,n\}$ and let $k$ be an integer such that $\alpha \leq k \leq \gamma.$
	The orbits for the action of $GL_n(F)$ on $Gr(k,F^n) \times Gr(r-k,F^n)$  are parameterized by an integer $l$ satisfying $\alpha \leq l \leq \min\{k,r-k\}.$		
\end{lemma}
	\begin{proof} 
	The group $ GL_n(F)$  acts on  $Gr(k,F^n) \times Gr(r-k,F^n)$
	by $g\cdot (W,W')=(g(W),g(W')).$ Suppose $W \in  Gr(k,F^n), W' \in Gr(r-k,F^n)$ and $g\in  GL_n(F)$ are such that  $g\cdot(W,W')=(W_1,W_1').$ Then, $\dim(W\cap W')=\dim(g(W\cap W'))=\dim(g(W)\cap g(W'))=\dim(W_1\cap W_1').$
	
	Conversely, suppose $(W_1,W_1')\in Gr(k,F^n) \times Gr(r-k,F^n)$ satisfies $\dim(W\cap W')=l=\dim(W_1\cap W_1').$ We shall show that there exists $g\in GL_n(F)$ such that $g\cdot (W,W')=(W_1,W_1').$ To this end, let $\mathcal{B}=\{u_1,\dots, u_l\}$ and $\mathcal{B}'=\{v_1,\dots, v_l\}$ be bases of $W\cap W'$ and $W\cap W_1'$ respectively. Extend $\mathcal{B}$ to a basis $\mathcal{C}$ of $W$ and a basis $\mathcal{D}$ of $W'.$ Write $\mathcal{C}=\mathcal{B}\cup \{u_{l+1},\dots, u_k\}$ and $\mathcal{D}=\mathcal{B}\cup \{u_{k+1},\dots, u_{k+(r-k-l)}\}.$ Similarly,  extend the $\mathcal{B}'$ to a basis $\mathcal{C}'$ of $W_1$ and a basis $\mathcal{D}'$ of $W_1'.$ Also write, $\mathcal{C}'=\mathcal{B}'\cup \{v_{l+1},\dots, v_k\}$ and $\mathcal{D}'=\mathcal{B}'\cup \{v_{k+1},\dots, v_{k+(r-k-l)}\}.$ We claim that $\{u_1,\dots, u_l, u_{l+1}, \dots, u_k, \dots, u_{r-l}\}$ is linearly independent. Accepting the claim, we may extend this linearly independent set to a basis $\{u_1,\dots, u_n\}$ of $F^n.$ In the same vein, we may extend $\{v_1,\dots, v_{r-l}\}$ to a basis $\{v_1,\dots, v_n\}$ of $F^n.$ Define $g$ by setting $u_i\mapsto v_i $ to obtain a $g\in GL_n(F)$ such that $g\cdot(W,W')=(W_1,W_1').$
	To prove the claim, assume that $\sum_{i=1}^{l} a_iu_i+\sum_{i=l+1}^{k} b_iu_i+\sum_{i=k+1}^{r-l} c_iu_i=0.$ We then have, $\sum_{i=1}^{l} a_iu_i+\sum_{i=l+1}^{k} b_iu_i=-\sum_{i=k+1}^{r-l} c_iu_i=x,$ say. But then, $x\in W\cap W'$ forcing all $a_i$'s, $b_i$'s and $c_i$'s to be zero.
	
	If $l=\dim(W\cap W'),$ it is easy to see  that $\max\{0,r-n\}\leq l \leq \min\{k,r-k\}$ by arguments similar to those in the proof of  Lemma \ref{boundsofkandl}(1).   \end{proof}

\subsubsection{Double coset representatives for $S\backslash P/J_{k,\alpha}$}
	
\begin{proposition} \label{doublecosetinsteps} Let $r$ and $n$ be integers such that $1\leq r<2n.$ Put $\alpha=\max\{0,r-n\}$ and $\gamma=\min\{r,n\}.$ 
		Let $k$ be an integer such that $\alpha\leq k \leq \gamma.$ The double coset representatives for  $S\backslash P_{n,n}/J_{k,\alpha}$ are parametrized 
		by an integer $l$ satisfying $\alpha\leq l \leq \min\{k,r-k\}.$
	\end{proposition}
	\begin{proof}
		Note that $S=\Delta GL_n(F) \ltimes N, P=(GL_n(F)\times GL_n(F)) \ltimes N$	and $J_{k,\alpha}= (P_{k,n-k}\times \tilde{P}_{r-k,n-r+k})\ltimes N_0$ where $\tilde{P}_{r-k,n-r+k}$ is given by \eqref{Ptilde}. One observes that $S\backslash P_{n,n}/J_{k,\alpha}$ can be identified with (cf.  \cite[proof of Proposition 7.1]{GS}) $\Delta GL_n(F) \backslash GL_n(F)\times GL_n(F)/ P_{k,n-k}\times \tilde{P}_{r-k,n-r+k}.$   Since $GL_n(F)/P_{k,n-k}$ and $GL_n(F)/\tilde{P}_{r-k,n-r+k}$ can be identified with $Gr(k,F^n)$ and $ Gr(r-k,F^n)$ respectively, by Lemma \ref{Goribts-doublecosets}, the double coset space $\Delta GL_n(F) \backslash GL_n(F) \times GL_n(F)/ P_{k,n-k}\times \tilde{P}_{r-k,n-r+k}$ is given by the orbits of the action of $\Delta GL_n(F)\simeq GL_n(F)$ on $Gr(k,F^n) \times Gr(r-k,F^n).$  Our statement now follows from Lemma \ref{lalternate}.
	\end{proof}
	\begin{corollary}\label{wklintermsofwkalpha}
	Let $r$ and $n$ be integers such that $1\leq r<2n.$	Put $\alpha=\max\{0, r-n\}$ and $\gamma=\min\{r,n\}.$ For each integer $k$ satisfying $\alpha\leq k \leq \gamma,$ let  $l$ be an integer such that $\alpha \leq l\leq \min\{k, r-k\}$ and put 
		
		\begin{equation}
			\sigma_{k,l}=\begin{pmatrix}
				I_k& 0& 0& 0& 0 &0 &0 &0\\
				0&I_{n-k}&0&0&0&0&0&0\\
				0& 0& I_{\alpha}&0&0&0&0&0\\
				0&0&0&0&0&I_{l-\alpha}&0&0\\
				0&0&0&I_{k-l}&0&0&0&0\\
				0&0&0&0&0&0&I_{r-(k+l)}&0\\
				0&0&0&0&I_{l-\alpha}&0&0&0\\
				0&0&0&0&0&0&0&I_{n-r+\alpha}
			\end{pmatrix}.
		\end{equation}
		Then,
		\begin{enumerate}
			\item for each $r$ such that $1\leq r <2n$, $\{w_{k,\alpha}: \alpha\leq k \leq  \gamma\}$ forms a complete set of double coset representatives for  the space $P\backslash GL_{2n}(F)/ P_{r,2n-r},$ 
			\item for each $k$ such that $\alpha \leq k \leq \gamma,$ $\{\sigma_{k,l}: \alpha\leq l\leq \min\{k, r-k\}\}$ forms a complete set of double coset representatives for the space $S\backslash P/ J_{k,\alpha}$ and, 
			\item for each $k$ and $l$ satisfying $\alpha\leq k \leq \gamma$ and $\alpha \leq l \leq \min\{k,r-k\},$ we have  $\sigma_{k,l} \cdot  w_{k,\alpha}=w_{k,l}.$
		\end{enumerate}
	\end{corollary}
	\begin{proof}
 Put $V_0=\langle e_1,\dots,e_k\rangle$, $V_1=\langle e_1,\dots, e_{\alpha}, e_{k+1},\dots, e_{k+(r-k-\alpha)}\rangle$, $W_0=\langle e_1,\dots,e_n\rangle$ and  $W_0'=\langle e_1,\dots,e_r \rangle.$  Further put $W_{k, \alpha} = \langle e_1, \dots , e_k, e_{n+1}, \dots , e_{n+\alpha}, e_{n+k+1}, \dots, e_{n+k+ (r-(k+\alpha))}\rangle$ and $V_{k,l}=\langle e_1,\dots,e_l,e_{k+1},\dots,e_{k+(r-(k+l))}\rangle.$ With the identification of $GL_{2n}(F)/ P_{r,2n-r}$ with $Gr(r,F^{2n}),$ we observe that the orbit of an element $W\in Gr(r,2n)$ is determined by $k:=\dim(W\cap W_0)$ (cf. \cite[\S 1.6, pp. 171-172]{Zelevinsky}) and also that $\alpha \leq k \leq \gamma.$ Since $\dim(W_{k,\alpha}\cap W_0)=k$ and $w_{k,\alpha}(W_0')=W_{k,\alpha},$ $\{w_{k,\alpha}: \alpha\leq k \leq  \gamma\}$ forms a complete set of double coset representatives for $P\backslash GL_{2n}(F)/ P_{r,2n-r},$ proving (1).

Under the transitive action of $GL_n(F)$ on $Gr(k,F^n)$ and $Gr(r-k,F^n),$ one observes that the stabilizers in $GL_{n}(F)$ of $V_0$ and $V_1$ are $P_{k,n-k}$ and $\tilde{P}_{r-k,n-r+k}$ respectively.  Note that $\dim(V_0\cap V_{k,l})=l.$ Consider the action of $GL_n(F)\times GL_n(F)$ (embedded diagonally in  $GL_{2n}(F)$) on $Gr(k,F^n) \times Gr(r-k,F^n)$ given by $(g_1,g_2)\cdot (W_1,W_2)=(g_1(W_1), g_2(W_2))$ where $(g_1,g_2)\in GL_n(F)\times GL_n(F)$ and $(W_1,W_2)\in Gr(k,F^n) \times Gr(r-k,F^n).$ Then, $\sigma_{k,l}\in GL_n(F)\times GL_n(F)$  and   $\sigma_{k,l}\cdot(V_0,V_1)=(V_0,V_{k,l}).$ The proof of (2) follows by applying Proposition \ref{doublecosetinsteps}. Also, (3) follows immediately by a direct verification.
\end{proof}

\begin{remark}
We note that the fiber of the map $S\backslash GL_{2n}(F)/ P_{r,2n-r}\to P\backslash GL_{2n}(F)/ P_{r,2n-r}$ lying over $w_{k,\alpha}$ is $\{w_{k,l}: \alpha\leq l \leq \min\{ k,r-k\} \}$ by  Corollary \ref{wklintermsofwkalpha}.
\end{remark}

\subsection{Bruhat decomposition}
In this section, we shall prove Theorem \ref{doublecosets-symmetric-intro}.
\subsubsection{}
 For a positive integer $n,$ let $\mathcal{J}_n=\{1,\dots, n\}.$ Denote by $S_n$ the permutation group  of $\mathcal{J}_n.$  For $1\leq r <2n,$ put $\mathcal{K}_{2n-r}=\{r+1,\dots,2n-r\}$ so that $\mathcal{J}_{2n}=\mathcal{J}_{r} \sqcup \mathcal{K}_{2n-r}.$ Let $\Omega_r(2n)$ denote the set of all subsets of $\mathcal{J}_{2n}$ with cardinality $r.$ The group $S_{2n}$ acts transitively on $\Omega_r(2n)$ and the stabilizer under this action of  $\mathcal{J}_r$ is  $S_{r} \times S_{2n-r},$ where $S_{2n-r}$ is identified with the permutation group of $\mathcal{K}_{2n-r}.$ Let 
\begin{equation}
\Delta S_n=\left\{\sigma\in S_{2n}: \sigma(j)\in \mathcal{J}_n \mbox{ and } \sigma(n+j)=n+\sigma(j), \forall  j\in \mathcal{J}_n \right\}.
\end{equation}

\subsubsection{}The double coset representatives for $\Delta S_n\backslash S_{2n}/S_{r}\times S_{2n-r}$ is given by the orbits for the action of $\Delta S_n$ on $\Omega_r(2n).$ For a subset $A$ of $\mathcal{J}_{2n},$ let $|A|$ denote the cardinality of $A.$ Let $\eta:\mathcal{J}_n\to \mathcal{K}_n$ denote the map $\eta(j)=n+j$ for all $j\in \mathcal{J}_n.$ 

\subsubsection{}The following two lemma's are analogues of the results obtained in Lemma \ref{dimensionspreserved} and Lemma \ref{boundsofkandl}. The proofs of these lemmas  can be obtained  from those of Lemma \ref{dimensionspreserved} and \ref{boundsofkandl} respectively by replacing dimensions of the vector spaces appearing there with cardinality of certain subsets, sum of subspaces with union and quotients of subspaces with set difference.  We include  the proofs here for the sake of completeness. 

\begin{lemma}
For $A,B\in \Omega_{r}(2n),$ assume that there exists $\sigma\in \Delta S_n$ such that $\sigma(A)=B.$ Then,
\begin{enumerate}
	\item  $|A\cap \mathcal{J}_n|=|B\cap \mathcal{J}_n|$ and 
	\item $|\eta(A\cap \mathcal{J}_n) \cap A|=|\eta(B\cap \mathcal{J}_n) \cap B|$.
\end{enumerate}
\end{lemma}
	
\begin{proof}
Since $\sigma$ is a bijection, $|A\cap \mathcal{J}_n|=|\sigma(A\cap \mathcal{J}_n)|=|B\cap \mathcal{J}_n|,$ proving (1).  We claim that
\begin{equation}\label{observe}
	\sigma(\eta(A\cap \mathcal{J}_n))\subset \eta(B\cap \mathcal{J}_n).
\end{equation}
 Let $y\in \eta(A\cap \mathcal{J}_n).$ Write $y=\eta(x),x\in A\cap \mathcal{J}_n$ so that $y=x+n\in \mathcal{K}_n.$ It follows  that $\sigma(y)=\sigma(x+n)=\sigma(x)+n$ as $\sigma\in \Delta S_n.$ But $\sigma(x)+n=\eta(\sigma(x)).$ As $\sigma(A)=B,$ we have $\sigma(x)\in B\cap \mathcal{J}_n$ and consequently $\sigma(y)\in \eta(B\cap \mathcal{J}_n)$ establishing \eqref{observe}.  Now, $\sigma(\eta(A\cap \mathcal{J}_n)\cap A)=\sigma(\eta(A\cap \mathcal{J}_n))\cap \sigma(A)\subset  \eta(B\cap \mathcal{J}_n)\cap \sigma(A)=\eta(B\cap \mathcal{J}_n)\cap B.$ Since $\sigma$ is one-one, by  \eqref{observe}, $|\eta(A\cap \mathcal{J}_n) \cap A|\leq |\eta(B\cap \mathcal{J}_n) \cap B|.$ Replacing $A,B,\sigma$ by $B,A,\sigma^{-1}$ establishes the inequality in the reverse direction as well, completing the proof of (2).
 \end{proof}
\begin{lemma}
For $A\in \Omega_r(2n),$ let $k=|A\cap \mathcal{J}_n|$ and $l=|\eta(A\cap \mathcal{J}_n) \cap A|.$  Put $\alpha=\max\{0,r-n\}$ and $\gamma=\min\{r,n\}.$ Then, $\alpha\leq k \leq \gamma $ and $\alpha \leq l \leq \min\{k,r-k\}.$
\end{lemma}
\begin{proof}
It is clear that $|A\cup \mathcal{J}_n|=r+n-k\leq 2n$ yields $k\geq r-n$ and therefore $k\geq \alpha.$  Comparing $|A\cap \mathcal{J}_n|$ with $|A|$ and $|\mathcal{J}_n|$ gives $k\leq \gamma$ proving $\alpha\leq k \leq \gamma.$

Note that $|A\cap \mathcal{K}_n|=r-k, |A\cup \mathcal{K}_n|=n+k$  and $|\eta(A\cap \mathcal{J}_n)|=k.$ From this, we obtain that $|\eta(A\cap \mathcal{J}_n) \cup A|=k+r-l.$ As $\eta(A\cap \mathcal{J}_n) \subset \mathcal{K}_n$, we have $\eta(A\cap \mathcal{J}_n)\cup A \subset \mathcal{K}_n\cup A$  yielding $k+r-l\leq n+k$ or $l\geq r-n.$ We can conclude that $l\geq \alpha.$ On the other hand, it is easy to see that $l\leq \min\{k,r-k\}$ as $\eta(A\cap \mathcal{J}_n)\cap A$ is a subset of both $\mathcal{K}_n\cap A$ and $\eta(A\cap \mathcal{J}_n).$  \end{proof}

\subsubsection{}
The following proposition is the companion to Theorem \ref{PSpsidoublecosets}.
\begin{proposition}\label{symmetricmain}
Suppose $A,B\in \Omega_r(2n)$ are such that  $|A\cap \mathcal{J}_n|=|B\cap \mathcal{J}_n|$ and $|\eta(A\cap \mathcal{J}_n) \cap A|=|\eta(B\cap \mathcal{J}_n) \cap B|.$Then there exists $\sigma\in \Delta S_n$ such that $\sigma(A)=B.$ 
	\end{proposition}
	\begin{proof}
We first show that we can write $\mathcal{J}_{2n}=\{x_1,\dots, x_{2n}\}$ such that the following holds:
	\begin{enumerate}
			\item[(i)] $\mathcal{J}_n=\{x_1,\dots,x_n\}$ and  $\mathcal{K}_n=\{x_{n+1},\dots,x_{2n}\}$ 
			\item[(ii)] $A=\{x_1,\dots,x_k\}\cup \{x_{n+1},\dots,x_{n+l}\}\cup\{x_{(n+k)+1},\dots, x_{(n+k)+r-k-l}\}$  and
			\item[(iii)]   $x_{n+j}=x_j+n=\eta(x_j)$ for $1\leq j \leq n.$
\end{enumerate}
Since $|\eta(A\cap \mathcal{J}_n)\cap A|=l,$ we can find elements $x_1,\dots x_l\in A\cap \mathcal{J}_n$ such that $\eta(A\cap \mathcal{J}_n)\cap \mathcal{K}_n=\{\eta(x_1),\dots \eta(x_l)\}.$ Since $|A\cap \mathcal{K}_n|=r-k,$ there exists $z_1,\dots, z_{r-k-l}\in (A\cap \mathcal{K}_n )\setminus \{\eta(x_1),\dots \eta(x_l)\}$ so that  we can write $A\cap \mathcal{K}_n=\{\eta(x_1),\dots, \eta(x_l), z_{1}, \dots, z_{r-k-l}\}.$
Similarly, we can find $x_{l+1},\dots, x_k\in (A\cap \mathcal{J}_n)\setminus \{x_1,\dots, x_l\}$ so that we have  $A\cap \mathcal{J}_n=\{x_1,\dots, x_l, x_{l+1}, \dots x_k\}.$ Put $x_{k+j}=\eta^{-1}(z_j)$ for $1\leq j \leq r-k-l.$ We claim that $ \{x_1,\dots x_k\}\cap \{x_{k+1}, \dots, x_{k+(r-k-l)}\}=\emptyset.$
In other words, our claim is $A\cap \mathcal{J}_n \cap \{\eta^{-1}(z_1),\dots, \eta^{-1}(z_{r-k-l})\}=\emptyset.$ Suppose  $x\in A\cap \mathcal{J}_n$ and $x=\eta^{-1}(z_j)$ for some $j.$ Then, $\eta(x)\in \eta(A\cap \mathcal{J}_n)\cap A\cap \mathcal{K}_n\subset  \eta(A\cap \mathcal{J}_n)\cap A$ yielding $\eta(x)\in \{\eta(x_1),\dots, \eta(x_l)\}$, a contradiction, as $\eta(x)=z_j$ is chosen in such a way that $z_j\notin \{\eta(x_1),\dots, \eta(x_l)\}.$ We have established that $\{x_1\dots,x_k,x_{k+1},\dots, x_{k+(r-k-l)}\}$ is a subset of $\mathcal{J}_n$ with cardinality $r-l.$ As $A$ is the disjoint union of $A\cap \mathcal{J}_n$ and $A\cap \mathcal{K}_n$ we can write
		$$A=\{x_1,\dots, x_k, \eta(x_1),\dots, \eta(x_l), \eta(x_{k+1}),\dots, \eta(x_{k+(r-k-l)})\}.$$
Also, it is clear that $\eta(A\cap \mathcal{J}_n)=\{\eta(x_1), \dots, \eta(x_k)\}.$ We can now find $n-r+l$ elements $x_{r-l+1},\dots, x_{n}\in \mathcal{J}_n\setminus \{x_1\dots,x_k,x_{k+1},\dots, x_{r-l}\}.$ Then, as $\eta$ is one-one
 $$\{\eta(x_{r-l+1}), \dots, \eta(x_n)\}\in \mathcal{K}_n\setminus (\eta(A\cap \mathcal{J}_n) \cup A\cap \mathcal{K}_n).$$
We have shown that $\mathcal{K}_n=\{\eta(x_1),\dots \eta(x_k), \eta(x_{k+1}),\dots, \eta(x_{r-l}),\eta(x_{r-l+1}),\dots, \eta(x_n)\}.$ Our claims (i), (ii) and (iii) are established.

By what we have proved, we can also write $\mathcal{J}_{2n}=\{x_1',\dots, x_{2n}'\}$ such that  $\mathcal{J}_n=\{x_1',\dots,x_n'\}$ and  $\mathcal{K}_n=\{x_{n+1}',\dots,x_{2n}'\},$ 
	$B=\{x_1',\dots,x_k'\}\cup \{x_{n+1}',\dots,x_{n+l}'\}\cup\{x_{(n+k)+1}',\dots, x_{(n+k)+r-k-l}'\}$  and
  $x_{n+j}'=x_j'+n=\eta(x_j')$ for $1\leq j \leq n.$ We may then define $\sigma:\mathcal{J}_{2n}\to \mathcal{J}_{2n}$ by $\sigma(x_j)=x_j'$ for $1\leq j \leq 2n.$ Clearly $\sigma\in S_{2n}$ and $\sigma(A)=B.$ If $x\in \mathcal{J}_n,$ it is obvious that $\sigma(x)\in \mathcal{J}_n.$ Suppose $x=x_j$ for some $j$ such that $1\leq j \leq n.$ Then, $\sigma(n+x)=\sigma(n+x_j)=\sigma(\eta(x_j))=\eta(x_j')=x_j'+n=\sigma(x_j)+n=\sigma(x)+n.$ Thus, $\sigma$ belongs to $\Delta S_n$ as required.
\end{proof}
\subsubsection{}	
Recall that we defined  $w_{k,l}'\in S_{2n}$ as follows: $w_{k,l}'$ maps $j\mapsto j$ for $1\leq j \leq k$ and $n+r-l+1\leq j\leq 2n,$ $k+j \mapsto n+j$ for $1\leq j\leq l,$ $k+l+j\mapsto n+k+j$ for $1\leq j \leq r-(k+l),$  $r+j\mapsto k+j$ for $1\leq j\leq n-k$ and $n+r-k+j\mapsto n+k+j$ for $1\leq j\leq k-l.$  If we regard $S_{2n}$ as a subgroup of $GL_{2n}(F)$  as permutation matrices, the permutation $w_{k,l}'$ corresponds to the matrix $w_{k,l}.$  The final result of our article is the following. 

\begin{theorem}\label{doublecosets-symmetric}
Let $n$ and  $r$ be integers such that $1\leq r<2n.$  Put $\alpha=\max\{0,r-n\}$ and $\gamma=\min\{r,n\}.$ For integers $k$ and $l$ such that $\alpha \leq k \leq \gamma$ and $\alpha\leq l \leq \min\{k,r-k\}$ define $w_{k,l}\in GL_{2n}(F)$ by \eqref{definitionwkl}.
Then, $w_{k,l}\in S_{2n}$  and $\left\{w_{k,l}: \alpha\leq k \leq \gamma, \alpha\leq l \leq \min\{k,r-k\} \right\}$
is a complete set of $(\Delta S_n ,S_{r}\times S_{2n-r})$-double coset representatives in $S_{2n}.$ In particular, we have a bijection 
$S\backslash GL_{2n}(F) /P_{r,2n-r} \overset{1:1}{\longleftrightarrow}\Delta S_n\backslash S_{2n}/S_r \times S_{2n-r}.$
\end{theorem}

\begin{proof} By Proposition \ref{symmetricmain}, the orbit of an element $A\in \Omega_r(2n)$ for the action of $\Delta S_n$ on $\Omega_r(2n)$ is determined by $|A\cap \mathcal{J}_n|$ and $|\eta(A\cap \mathcal{J}_n)\cap A.$
For each $k,l$ satisfying $\alpha \leq k \leq \gamma, \alpha\leq l \leq \min\{k,r-k\},$ put $A_{k, l} = \{1, \dots , k, n+1, \dots , n+l, n+k+1, \dots, n+k+ (r-(k+l))\}.$ Clearly, $A_{k,l}\cap \mathcal{J}_n=\{1,\dots,k\}, \eta(A_{k,l}\cap \mathcal{J}_n)=\{n+1,\dots, n+k \}$ and $A\cap \mathcal{K}_n=\{ n+1, \dots , n+l, n+k+1, \dots, n+k+ (r-(k+l))\}.$ Hence, $\eta(A_{k,l}\cap \mathcal{J}_n)\cap A=\{ n+1,\dots, n+l\}.$  In view of this, if $(k,l)\neq (k',l')$ it is easy to see that $A_{k,l}\neq A_{k',l'}.$ Under the identification of $S_{2n}/S_{r}\times S_{2n-r}$ with $\Omega_r(2n),$ the left coset $ \sigma (S_r\times S_{2n-r})$ corresponds to the subset $\sigma(\mathcal{J}_r).$  Let $w_{k,l}$ be as in \eqref{definitionwkl}. Clearly, $w_{k,l}\in S_{2n}.$ Since $w_{k,l}(\mathcal{J}_r)=A_{k,l},$  $\{A_{k,l}: \alpha \leq k \leq \gamma, \alpha \leq l \leq \min\{k,r-k\} \}$ is a complete set of orbit representatives for the action of $\Delta S_{n}$ on $S_{2n}/S_{r}\times S_{2n-r}.$  Consequently, $\left\{w_{k,l}: \alpha\leq k \leq \gamma, \alpha\leq l \leq \min\{k,r-k\} \right\}$
		is a complete set of $(\Delta S_n ,S_{r}\times S_{2n-r})$-double coset representatives in $S_{2n}.$  \end{proof}

\section*{Acknowledgements}
This work forms a part of the PhD thesis of C. Harshitha at the Indian Institute of Science Education and Research Tirupati and she acknowledges the financial support provided by the National Board of Higher Mathematics (NBHM), Govt of India and  Prime Minister's Research Fellowship (PMRF ID 0900452) program, Govt of India during this work. The authors thank Dipendra Prasad for suggesting the coordinate free approach undertaken in the article and for suggesting that a result such as Theorem \ref{doublecosets-symmetric-intro} shall hold. The authors also thank him for constant encouragement and numerous helpful conversations  during the preparation of this article.


\begin{thebibliography}{99}
 	
 	\bibitem{KHJNT}
 	K. Balasubramanian, A. Dangodara, H. Khurana,
 {\it On a twisted Jacquet module of \rm{GL}(2n) over a finite field}, J. Number Theory, {\bf 271}(2025), 458--474.
 	
 \bibitem{BZ1}
 I. N. Bernstein, A. V. Zelevinsky, {\it Representations of \rm{GL}(n,F) where F is a non-archimedean local field}, Russian Math. Surveys, {\bf 31:3} (1976), 1--68.
 	
 	
 	\bibitem{BZ2}
 I. N. Bernstein,  A. V. Zelevinsky, {\it Induced representations of reductive p-adic groups I}, Ann. Scient. E.N.S., 4e serie, t., {\bf 10} (1977), 441--472.
 
 \bibitem{BD}
 P. Blanc, P. Delorme, {\it Vecteurs distributions H-invariants de repr\'esentations induites, pour un
 	espace sym\'etrique r\'eductif $p$-adique $G/H$}, Ann. Inst. Fourier, Grenoble, Tome 58, no 1 (2008), p. 213-261.
 		
 	\bibitem{GS}
 	W. T. Gan, S. Takeda, {\it On Shalika Periods and a theorem of Jacquet-Martin}, Amer. J. Math. {\bf 132} (2010), 475--528.
 	
 \bibitem{Garret}
 P. Garrett, {\it Buildings and Classical Groups}, Chapman and Hall, London, 1997.
 
 \bibitem{Geng}
 Z. Geng, {\it On the existence of twisted Shalika periods: the archimdean case}, \url{http://arxiv.org/abs/2501.11917v1}
 
 \bibitem{HV2}
 C. Harshitha, C. G. Venketasubramanian,  {\it Structure of twisted Jacquet modules of principal series representations of $GL_{2n}(F)$}, preprint.
 	
 	\bibitem{DPDegIMRN}
 	D. Prasad, {\it The space of degenerate Whittaker models for general
 		linear groups over a finite field}, Int. Math. Res. Not. {\bf 11} (2000), 579--595.
 		
 	\bibitem{DPDegTIFR}
 		D. Prasad, {\it  The space of degenerate Whittaker models for \rm{GL(4)} over p-adic fields}, in Cohomology of arithmetic groups, L-functions and
 		automorphic forms, (Mumbai, 1998/1999), Tata. Inst. Fund. Res. Stud. Math., {\bf 15}, Tata Inst. Fund. Res., Bombay, (2001), 103--115.
 		
 	\bibitem{Neretin}
 	Y.A. Neretin, {\it On double cosets of groups {\rm GL(n)} with respect to subgroups of block strictly triangular matrices}, Linear and Multilinear Algebra, {\bf 70:19}, 4620--4632.
 	 \bibitem{SV}
 	S. K. Pandey, C. G. Venketasubramanian, {\it Structure of Twisted Jacquet modules of principal series representations of $Sp_4(F)$}, J. Pure Appl. Algebra {\bf 229} (2025), Paper No. 108060, 37 pp.
 	\bibitem{Serre}
 	J.P. Serre, {\it Linear representations of finite groups}, Graduate Texts in Mathematics, vol. 42, Springer-Verlag, 1977.
 	
 		\bibitem{Zelevinsky}
 	A. Zelevinsky, {\it Induced representations of reductive p-adic groups. II. On
 		irreducible representations of GL(n)}, Annales scientifiques de l’É.N.S. 4 e série, tome 13, no 2 (1980), p. 165-210.

\end{thebibliography}
\end{document}